\documentclass{IEEEtran}
\usepackage{cite}
\usepackage{amssymb,amsthm,amsmath,amsfonts,bm,xcolor}
\usepackage{algorithmic}
\usepackage{graphicx}
\usepackage{textcomp}
\newtheorem{theorem}{Theorem}

\newtheorem{corollary}[theorem]{Corollary}
\newtheorem{lemma}{Lemma}
\newtheorem{remark}{Remark}

\newtheorem{assumption}{Assumption}
\def\BibTeX{{\rm B\kern-.05em{\sc i\kern-.025em b}\kern-.08em
T\kern-.1667em\lower.7ex\hbox{E}\kern-.125emX}}
\newcommand{\diag}{\mbox{diag}}
\begin{document}
\title{Boundary control of multi-dimensional discrete-velocity kinetic models}
\author{Haitian Yang, Wen-An Yong
\thanks{Haitian Yang is with 
the Department of Mathematical Sciences, Tsinghua University, Beijing 100084, China (e-mail: yht21@mails.tsinghua.edu.cn).}
\thanks{Wen-An Yong is with the Department of Mathematical Sciences, Tsinghua University, Beijing 100084, China and with Beijing Institute of Mathematical Sciences and Applications, Beijing 101408, China (e-mail: wayong@tsinghua.edu.cn).}}

\maketitle

\begin{abstract}
	This technical note is concerned with boundary stabilization of multi-dimensional discrete-velocity kinetic models. By exploiting a certain stability structure of the models and adapting an appropriate Lyapunov functional, we derive feasible control laws so that the corresponding solutions decay exponentially in time. The result is illustrated with an application to the two-dimensional coplanar model in a square container. The effectiveness of the derived control laws is confirmed by numerical simulations.
\end{abstract}


\section{Introduction}
We are interested in boundary stabilization of discrete velocity models of the Boltzmann equation in the kinetic theory \cite{Gatignol1975}. This kind of models was first suggested by James C. Maxwell and considered again in Carleman’s book \cite{carleman1957problemes,platkowski1988discrete}. They are systems of multi-dimensional first-order hyperbolic equations with source terms (balance laws) and can be used to simulate behaviours of
rarefied gases \cite{inamuro1990numerical}, to investigate the shock structure \cite{gatignol1975kinetic}, to study the Couette and Rayleigh flow \cite{Gatignol1975}, etc. 
The growing interest thereof during the last fifty years parallels the increasing activity in the theory of the Boltzmann equation itself. The connection between these models and the Boltzmann equation was studied in \cite{palczewski1997consistency}. 

Over the past two decades, the boundary control problem of hyperbolic balance laws has attracted much attention in the mathematical and engineering community due to wide applications. In \cite{li2010strong}, exact boundary controllability for quasi-linear systems was shown via the characteristics method. A general result using a Lyapunov function with a smallness assumption on the source terms can be found in \cite{bastin2016stability,DIAGNE2012109}. 
In \cite{Hu2016,Hu2019}, general linear or quasi-linear coupled systems were treated with the backstepping method. We note that the above works are mainly for the spatially one-dimensional problems with non-characteristic boundaries. We also refer to \cite{AURIOL2016300,BASTIN201766,coron2022lyapunov,DEUTSCHER201754,HAYAT201952,yu2022traffic} and the references cited therein for related state-of-the-art research. For characteristic boundaries, some progresses have been made in  \cite{de2022backstepping,YONG2019252} recently.

As in one-dimensional problems, source terms can cause instability in the overall system \cite{yu2022traffic}. Similar instability can be expected for multi-dimensional problems.
% Thus the role of the source terms will be critical to the boundary control problems. 
To the best of our knowledge, there are only few works about multi-dimensional problems with source terms \cite{herty2022stabilization,xu2002exponential,yang2023feedback}.
In \cite{xu2002exponential}, boundary stabilization was achieved for symmetric hyperbolic systems under certain ad hoc assumptions (see formulas (2.28) and (2.32) therein). In \cite{herty2022stabilization}, the authors imposed strict assumptions on the source terms (see formulas (2.3) and (2.6) therein). The assumptions do not seem straightforward to check whether they are true or not. On the other hand, our recent work \cite{yang2023feedback} uses a physically relevant dissipation structure \cite{yong1999singular} to achieve boundary stabilization of the two-dimensional Saint-Venant equations.

The aim of this technical note is to extend the Lyapunov function method \cite{bastin2016stability,HERTY201612} and investigate the multi-dimensional discrete velocity models by exploiting the dissipation structure of the systems. This dissipation structure is characterized by the  structural stability condition proposed in \cite{yong47basic,yong1999singular,yong2008interesting}. As shown in \cite{yong2008interesting}, the structure defines a wide class of physically relevant problems, has its root in non-equilibrium thermodynamics, and is closely related to the celebrated Onsager reciprocal relation \cite{PhysRev.37.405,PhysRev.38.2265}. We refer to \cite{HERTY201612,WANG2020104815,YONG2019252} for other recent works exploiting the structural stability condition to achieve boundary stabilization in the one-dimensional case.

The technical note is organized as follows. We first show that the discrete velocity models satisfy the structural stability condition. Then we generalize the modified Lyapunov function in \cite{bastin2016stability,HERTY201612} which takes full advantage of the dissipation structure of the systems to achieve boundary stabilization for the multi-dimensional discrete velocity models. Finally, we take the two-dimensional coplanar model \cite{platkowski1988discrete} as an example to illustrate that numerous boundary control laws can be chosen according to the practical situations, as stated in \cite{yang2023feedback}.

\section{Preliminaries}	\label{section2}
Consider a gas of identical particles moving within a bounded domain $\Omega \subset \mathbb{R}^d$ with piecewise smooth boundary. We postulate that each particle can move only with one of the velocities $u_i$ in a given finite set $S \subset \mathbb{R}^d$:
$$
S=\{ u_k \in \mathbb{R}^d: \ k=1,\cdots,n\}.
$$
We denote by $\tilde{f}_k=\tilde{f}_k(t,x)$ the number density of gas particles moving with the constant
velocity $u_k\in S$ at time $t$ ($t\geq 0$) and position $x=(x_1,\cdots,x_d) \in \bar{\Omega}$, and by $\tilde{f}=(\tilde{f}_1,\cdots,\tilde{f}_n)^T.$ Here and below, the superscript $^T$ is the transpose of a vector or matrix. 

In the absence of external forces, the governing equation for each $\tilde{f}_k$ is \cite{Gatignol1975,platkowski1988discrete}
$$
\partial_t \tilde{f}_{k}+u_k\cdot \nabla_x \tilde{f}_k=Q_k(\tilde{f}), \quad k=1,\cdots,n,
$$
where $u_k\cdot \nabla_x$ denotes the differential operator $\sum_{j=1}^d u_{kj}\partial_{x_j}$ with $u_{kj}$ the $j$-th component of the vector $u_k$. For binary collisions, the source term $Q_k(\tilde{f})$ is given by
\begin{equation*}
	Q_k(\tilde{f})=\sum_{ijl} (A_{ij}^{kl}\tilde{f}_i\tilde{f}_j-A_{kl}^{ij}\tilde{f}_k\tilde{f}_l),
\end{equation*}
where the summation is taken over all $i,j,l \in \{1,2,\cdots,n\}$ and the non-negative constant coefficients $A_{ij}^{kl}$ stand for the transition rate in collision and satisfy
\begin{equation} \label{2.1}
	A_{ij}^{kl}=A_{kl}^{ij}=A_{lk}^{ij} \geq 0.
\end{equation} 

Denote $(Q_1(\tilde{f}),\cdots,Q_n(\tilde{f}))^T$ by $Q(\tilde{f}).$ Then the  kinetic model above can be rewritten in its vector form:
\begin{equation} \label{2.2}
	\tilde{f}_t+\sum_{j=1}^d \begin{pmatrix}
		u_{1j} & 0 & \cdots & 0 \\
		0 & u_{2j} & \cdots & 0 \\
		\vdots & \vdots &\ddots & \vdots \\
		0 & 0 & \cdots & u_{nj}
	\end{pmatrix}\tilde{f}_{x_j}=Q(\tilde{f}).
\end{equation}
Here and below, the subscripts $t$ and $x_j$ denote the corresponding
partial derivatives. This is a semi-linear first-order hyperbolic system with a source term. On physical grounds, the solution components of (\ref{2.2}) should be non-negative since $\tilde{f}_k$ is a number density. Thus we introduce the following cone as the state space
$$
G=\{\tilde{f}\in \mathbb{R}^n: \ \tilde{f}_k>0, \ k=1,\cdots,n\}.
$$
For the balance law (\ref{2.2}), a constant vector $f_e \in G \subset \mathbb{R}^n$ is called a uniform steady state \cite{bastin2016stability} if it satisfies  $Q(f_e)=0$.

With the symmetry property (\ref{2.1}), it is shown in Section IX of \cite{yong2008interesting} that the source term can be written as
\begin{equation*} 
	Q(\tilde{f})=-\mathcal{L}(\tilde{f})\eta_{\tilde{f}}(\tilde{f})\equiv -\mathcal{L}(\tilde{f})\begin{pmatrix}
		\log \tilde{f}_1 \\
		\vdots \\
		\log \tilde{f}_n
	\end{pmatrix}, \qquad \tilde{f} \in G,
\end{equation*}
where the $n \times n$-matrix $\mathcal{L}(\tilde{f})$ is symmetric positive semi-definite and has a $\tilde{f}$-independent null space, 
$$
\eta(\tilde{f})=\sum_{k=1}^n \tilde{f}_k(\log \tilde{f}_k-1)
$$
is an entropy function, and $\eta_{\tilde{f}}(\tilde{f})$ denotes the gradient of $\eta(\tilde{f})$. 
Remarkably, the symmetry of $\mathcal{L}(\tilde{f})$ corresponds to the celebrated Onsager reciprocal relation in non-equilibrium thermodynamics \cite{PhysRev.37.405,PhysRev.38.2265,yong2008interesting}. And the $\tilde{f}$-independence expresses the fact that the physical laws of conservation hold true, no matter what state the underlying thermodynamical system is in equilibrium, nonequilibrium, and so on \cite{yong2008interesting}.

Next, we present the following lemma implicitly proved in \cite{yong2008interesting}, showing that the system (\ref{2.2}) satisfies the structural stability condition proposed in \cite{yong47basic,yong1999singular} for general nonlinear multi-dimensional hyperbolic relaxation systems.
\begin{lemma} \label{L2.1}
	For a given steady state $f_e \in G,$ there exists an invertible matrix $P$  and a diagonal positive definite matrix $\Lambda_0$ such that
	\begin{equation} \label{2.3}
		PQ_{\tilde{f}}(f_e)P^{-1}=-\begin{pmatrix}
			0 & 0 \\
			0 & \Lambda
		\end{pmatrix}
	\end{equation}
	and
	\begin{equation} \label{2.4}
		\qquad \Lambda_0 Q_{\tilde{f}}(f_e)=-P^T\begin{pmatrix}
			0 & 0 \\
			0 & \Lambda
		\end{pmatrix} P,
	\end{equation}
	where $\Lambda \in \mathbb{R}^{r\times r}$ is a symmetric positive definite matrix with $(n-r)$ the dimension of the null space of $\mathcal{L}(f_e).$ 
\end{lemma}

\begin{proof}
	
	By the definition of the steady state $f_e \in G,$ we have $Q(f_e)=-\mathcal{L}(f_e)\eta_{\tilde{f}}(f_e)=0,$ meaning that $\eta_{\tilde{f}}(f_e)$ is in the null space of $\mathcal{L}(f_e)$. From the aforementioned $\tilde{f}$-independence of the null space of $\mathcal{L}(\tilde{f})$, it follows that $\eta_{\tilde{f}}(f_e)$ is also in the null space of $\mathcal{L}(\tilde{f})$ for any $\tilde{f} \in G.$ Thus we conclude that 
	\begin{equation*} 
		\mathcal{L}(\tilde{f})\eta_{\tilde{f}}(f_e)=0, \qquad \forall \ \tilde{f} \in G,
	\end{equation*}
	and thereby
	\begin{equation*} 
		Q(\tilde{f})=-\mathcal{L}(\tilde{f})\eta_{\tilde{f}}(\tilde{f})=-\mathcal{L}(\tilde{f})[\eta_{\tilde{f}}(\tilde{f})-\eta_{\tilde{f}}(f_e)].
	\end{equation*}
	Differentiating this equality at $\tilde{f}=f_e,$ we obtain
	\begin{equation*} 
		Q_{\tilde{f}}(f_e)=-\mathcal{L}(f_e)\eta_{\tilde{f}\tilde{f}}(f_e).
	\end{equation*}
	Note that the Hessian matrix $\eta_{\tilde{f}\tilde{f}}(\tilde{f})$ is a diagonal and positive definite matrix:
	$$
	\eta_{\tilde{f}\tilde{f}}(\tilde{f})=\diag\left(\frac{1}{\tilde{f}_1},\cdots,\frac{1}{\tilde{f}_n}\right).
	$$
	
	Take 
	\begin{equation} \label{2.5}
		\Lambda_0=\eta_{\tilde{f}\tilde{f}}(f_e)=\diag\left(\frac{1}{f_1^e},\cdots,\frac{1}{f_n^e}\right)
	\end{equation}
	with $f_i^e$ the $i$-th component of $f_e$.
	Thanks to the symmetry of $\Lambda_0^{1/2}\mathcal{L}(f_e)\Lambda_0^{1/2}$, there exists an orthogonal matrix $H$ such that
	$$
	\Lambda_0^{1/2}\mathcal{L}(f_e)\Lambda_0^{1/2}=H^T\begin{pmatrix}
		0 & 0 \\
		0 & \Lambda
	\end{pmatrix}  H,
	$$
	with $\Lambda \in \mathbb{R}^{r \times r}$ a diagonal matrix with positive entries. Then we take $P=H\Lambda_0^{1/2}$ and deduce that
	\begin{equation*}
		\begin{aligned}
			&PQ_{\tilde{f}}(f_e)P^{-1}\\=&-P\mathcal{L}(f_e)\Lambda_0P^{-1}\\=&-H\Lambda_0^{1/2}\Lambda_0^{-1/2}H^{T}\begin{pmatrix}
				0 & 0 \\
				0 & \Lambda
			\end{pmatrix}H\Lambda_0^{-1/2} \Lambda_0 \Lambda_0^{-1/2}H^T\\=&-\begin{pmatrix}
				0 & 0 \\
				0 & \Lambda
			\end{pmatrix},
		\end{aligned}
	\end{equation*}
	and
	\begin{equation*}
		\begin{aligned}
			\Lambda_0Q_{\tilde{f}}(f_e)=-\Lambda_0\mathcal{L}(f_e)\Lambda_0&=-\Lambda_0^{1/2}H^T\begin{pmatrix}
				0 & 0 \\
				0 & \Lambda
			\end{pmatrix}H\Lambda_0^{1/2}\\&=-P^T\begin{pmatrix}
				0 & 0 \\
				0 & \Lambda
			\end{pmatrix}P.
		\end{aligned}
	\end{equation*}
	This completes the proof.
\end{proof}

Now we linearize the discrete velocity model (\ref{2.2}) around the steady state $f_e$ to obtain
\begin{equation} \label{2.6}
	f_t+\sum_{j=1}^d \Lambda_j    f_{x_j}=Q_{\tilde{f}}(f_e)f,
\end{equation}
with $f=\tilde{f}-f_e$ the deviation from the steady state.
Here $\Lambda_j$ is the diagonal constant matrix $\diag\left(u_{1j},u_{2j},\cdots,u_{nj}\right)$ and the matrix $Q_{\tilde{f}}(f_e)$ enjoys the properties stated in Lemma \ref{L2.1}. 

To solve the first-order hyperbolic system (\ref{2.6}) in the bounded domain $\Omega \subset \mathbb{R}^d$, we need to specify proper initial and boundary conditions. With these conditions, the existence and uniqueness of solutions to system (\ref{2.6}) can be established (see Appendix). 
Denote by $\mathbf{n}(x)=(n_1(x),\cdots,n_d(x))$  the unit outward normal vector at the boundary point $x=(x_1,\cdots,x_d) \in \partial \Omega.$ 

The components of $f$ corresponding to the negative (resp. positive) entries of the diagonal matrix
\begin{equation} \label{2.7}
	\sum_{j=1}^d n_j(x)\Lambda_j
\end{equation}
are referred to as incoming (resp. outgoing) variables at the boundary point. 
According to the classic theory \cite{Serre2006,higdon1986initial,Majda1975InitialboundaryVP,Russell1978}, the proper boundary condition specifies an algebraic relation to determine the incoming variables in terms of the outgoing variables.

We conclude this section with the following technical assumption:
\begin{assumption} \label{A2.1}
	The discrete velocity set $S$ does not contain the origin in $\mathbb{R}^d.$ 
\end{assumption}
\noindent This assumption ensures that the diagonal matrix $\sum_{j=1}^d \Lambda_j^2$ is positive definite. Physically, it means that the gas system under consideration does not contain static particles.

\begin{remark}
	In the one-dimensional case, i.e. $d=1,$ the last assumption is equivalent to the fact that the boundary is non-characteristic, that is, the matrix (\ref{2.7}) is invertible for each $x \in \partial \Omega$. For this kind of problems, many excellent works have been done, and in particular, the boundary control can be achieved with the backstepping method with arbitrary source terms \cite{Hu2016}. Otherwise, if the boundary is characteristic, the controllability requires that some restrictions on the source terms should be imposed, see \cite{de2022backstepping,yang2023feedback,YONG2019252}.
\end{remark}

\section{Main result} \label{section3}

Our goal is to construct boundary conditions so that the resultant initial-boundary-value problem has a unique $L^2$-solution $f=f(t,x)$ decaying in time exponentially. For this purpose, we follow \cite{bastin2016stability,HERTY201612} and introduce the following functional
\begin{equation*} 
	\begin{aligned}
		L(t)= &\alpha \int_{\Omega}  f^T(t,x) \Lambda_0 f(t,x) dx\\&+\int_{\Omega}  f^T(t,x)\exp\left(-\sum_{j=1}^d \Lambda_jx_j\right) f(t,x) dx.
	\end{aligned}
\end{equation*}
Here $\alpha$ is a positive constant to be determined, $\Lambda_0$ is the diagonal matrix defined in (\ref{2.5}), and $\exp(-\sum_{j=1}^d \Lambda_jx_j)$ is a diagonal matrix with entries $\exp(-\sum_{j=1}^d u_{ij}x_j)$ for $i=1,\cdots,n.$ With the bounded domain $\Omega$, this functional $L(t)$ is obviously equivalent to the square of the usual $L^2$-norm of $f$, namely
\begin{equation} \label{3.1}
	\lambda_{\min} \| f(t,\cdot)\|^2  \leq L(t) \leq \lambda_{\max} \| f(t,\cdot)\|^2.
\end{equation}
Here $\lambda_{\min}$ (resp.\ $\lambda_{\max}$) are the smallest (resp.\ biggest) eigenvalue of the symmetric positive definite matrix 
$$
\alpha \Lambda_0+\exp\left(-\sum_{j=1}^d \Lambda_jx_j\right)
$$ 
over $x \in \bar{\Omega}.$

% For a vector value function $\xi(t,x),$ the usual $L^2$ norm is 
% $$
% \| \xi(t)\|^2=\| \xi(t,\cdot) \|^2:=\int_{\Omega} 
% \xi^T(t,x) \xi(t,x) dx.
% $$

Our main result is 
\begin{theorem} \label{T3.1}
	Under Assumption \ref{A2.1}, there exist proper boundary conditions such that the corresponding solutions $f=f(t,x) \in C([0,\infty);L^2(\Omega))$ to system (\ref{2.6}) with initial data $f_0=f_0(x)\in L^2(\Omega)$ are exponentially stable. Namely, there exist positive constants $\nu$ and $C$ such that, for all $f_0 \in L^2(\Omega)$,
	\begin{equation*}
		\| f(t,\cdot) \| \leq C\exp(-\nu t) \| f_0 \|. 
	\end{equation*} \\
\end{theorem}
\vspace{-0.5cm}
% For the existence of $L^2$-solutions, see Appendix of this technical note.
The proof of this theorem relies on the following two lemmas, where 
$$
\begin{pmatrix}
	u \\ q
\end{pmatrix}:=Pf
$$
with the same partition as in Lemma \ref{L2.1}, that is, $u \in \mathbb{R}^{n-r},q \in \mathbb{R}^{r}.$

\begin{lemma}  \label{L3.1}
	The $L^2$-solutions to (\ref{2.6}) satisfy the following inequality
	\begin{equation*} 
		\frac{d}{dt}\int_{\Omega}  f^T \Lambda_0 f dx+\mathcal{BC}_1 \leq -2 \lambda \|q(t,\cdot)\|^2,
	\end{equation*}
	where
	$$
	\mathcal{BC}_1= \int_{\partial \Omega}  f^T\Lambda_0\left[\sum_{j=1}^d n_j(x)\Lambda_j\right]f   d\sigma
	$$ 
	with $\lambda$ the smallest eigenvalue of $\Lambda$ in Lemma \ref{L2.1}. 
\end{lemma}

\begin{proof}
	Referring to Lemma \ref{L2.1}, we multiply the system (\ref{2.6}) with $f^T\Lambda_0$ from the left to obtain 
	\begin{equation*}
		\begin{aligned}
			( f^T\Lambda_0f)_t+\sum_{j=1}^d (f^T\Lambda_0\Lambda_jf)_{x_j}&= f^T(\Lambda_0Q_{\tilde{f}}(f_e)+Q_{\tilde{f}}^T(f_e)\Lambda_0)f \\
			&=-2(Pf)^T\begin{pmatrix}
				0 & 0 \\
				0 & \Lambda
			\end{pmatrix}Pf \\
			&=-2 q^T\Lambda q \leq -2\lambda q^Tq.
		\end{aligned}
	\end{equation*}
	Integrating this inequality over $\Omega$ and using the divergence theorem gives the lemma.
\end{proof}

\begin{lemma} \label{L3.2}
	Under Assumption \ref{A2.1}, there exist two positive constants $C_1,C_2$ such that
	\begin{equation*} 
		\begin{aligned}
			&\frac{d}{dt} \int_{\Omega}  f^T\exp\left(-\sum_{j=1}^d \Lambda_jx_j\right) f dx+\mathcal{BC}_2 \\\leq& -\frac{C_1}{2}\| u(t,\cdot)\|^2+(C_2-C_1)\| q(t,\cdot)\|^2,
		\end{aligned}
	\end{equation*}
	where
	\begin{equation*}
		\mathcal{BC}_2=\int_{\partial \Omega} f^T\exp\left(-\sum_{j=1}^d \Lambda_jx_j\right)\left[\sum_{j=1}^d n_j(x)\Lambda_j\right]f  d\sigma.
	\end{equation*}
\end{lemma}

\begin{proof}
	Multiplying the system (\ref{2.6}) by $f^T\exp(-\sum_{j=1}^d \Lambda_jx_j)$ from the left, we obtain 
	\begin{equation} \label{3.2} 
		\begin{aligned} 
			&\left(f^T\exp\left(-\sum_{j=1}^d \Lambda_jx_j\right)f\right)_t\\&+\sum_{i=1}^d \left(f^T\exp\left(-\sum_{j=1}^d \Lambda_jx_j\right)\Lambda_i f\right)_{x_i} \\[4mm] 
			=&-f^T\left(\sum_{i=1}^d\Lambda_i^2\right)\exp\left(-\sum_{j=1}^d \Lambda_jx_j\right) f \\[4mm]&+f^T \Bigg[\exp\left(-\sum_{j=1}^d \Lambda_jx_j\right)Q_{\tilde{f}}(f_e)\\[4mm]&+Q^T_{\tilde{f}}(f_e)\exp\left(-\sum_{j=1}^d \Lambda_jx_j\right)\Bigg]f.
		\end{aligned}
	\end{equation}
	Under Assumption \ref{A2.1}, the diagonal matrix 
	$$
	\left(\sum_{i=1}^d\Lambda_i^2\right)\exp\left(-\sum_{j=1}^d \Lambda_jx_j\right)
	$$
	is positive definite. Therefore, the first term in the right-hand side of (\ref{3.2}) can be estimated as
	\begin{equation*}
		\begin{aligned}
			&f^T\left(\sum_{i=1}^d\Lambda_i^2\right)\exp\left(-\sum_{j=1}^d \Lambda_jx_j\right) f\\[4mm]=& (Pf)^TP^{-T}\left(\sum_{i=1}^d\Lambda_i^2\right)\exp\left(-\sum_{j=1}^d \Lambda_jx_j\right)P^{-1} (Pf) \\[4mm]
			=&\begin{pmatrix}
				u \\
				q
			\end{pmatrix}^TP^{-T}\left(\sum_{i=1}^d\Lambda_i^2\right)\exp\left(-\sum_{j=1}^d \Lambda_jx_j\right)P^{-1}\begin{pmatrix}
				u \\
				q
			\end{pmatrix}\\[4mm]
			\geq& C_1u^Tu+C_1q^Tq,
		\end{aligned}
	\end{equation*}
	where $C_1$ is the smallest eigenvalue of the symmetric positive definite matrix 
	$$
	P^{-T}\left(\sum_{i=1}^d\Lambda_i^2\right)\exp\left(-\sum_{j=1}^d \Lambda_jx_j\right)P^{-1}
	$$ over $x \in \bar{\Omega}.$

	For the second term, we set 
	$$
	\mu(x):=P^{-T}\exp\left(-\sum_{j=1}^d \Lambda_jx_j\right)P^{-1}=\begin{pmatrix}
		\mu_{11}(x) & \mu_{12}(x) \\
		\mu_{21}(x) & \mu_{22}(x)
	\end{pmatrix},
	$$
	where the submatrices $\mu_{11}(x) \in \mathbb{R}^{(n-r)\times (n-r)}$ and $\mu_{22}(x)\in \mathbb{R}^{r \times r}.$ By (\ref{2.3}), the second term can be estimated as
	\begin{equation*}
		\begin{aligned}
			&f^T \left[P^T\mu(x)PQ_{\tilde{f}}(f_e)+Q^T_{\tilde{f}}(f_e)P^{T}\mu(x)P\right]f \\[4mm]
			=& (Pf)^T\left[\mu(x)PQ_{\tilde{f}}(f_e)P^{-1}+P^{-T}Q^T_{\tilde{f}}(f_e)P^{T}\mu(x)\right](Pf)\\[4mm]
			=&-\begin{pmatrix}
				u \\ q
			\end{pmatrix}^T \begin{pmatrix}
				0 & \mu_{12}(x)\Lambda \\
				\Lambda\mu_{21}(x) & \mu_{22}(x)\Lambda+\Lambda\mu_{22}(x)
			\end{pmatrix}\begin{pmatrix}
				u \\ q
			\end{pmatrix}\\[4mm]
			\leq& \frac{C_1}{2}u^Tu+C_2q^Tq.
		\end{aligned}
	\end{equation*}
	Here Young's inequality has been used and the positive constant $C_2$ depends on  $C_1.$ 
	
	Combining the two estimates above, we arrive at
	\begin{equation*}
		\begin{aligned}
			&\left(f^T\exp\left(-\sum_{j=1}^d \Lambda_jx_j\right)f\right)_t\\&+\sum_{i=1}^d \left(f^T\exp\left(-\sum_{j=1}^d \Lambda_jx_j\right)\Lambda_i f\right)_{x_i} \\[3mm]
			\leq& -\frac{C_1}{2}u^Tu+(C_2-C_1)q^Tq.
		\end{aligned}
	\end{equation*}
	Integrating this inequality over $\Omega$ and using the divergence theorem gives the lemma.
	
\end{proof} 

\begin{remark}
	The proofs of Lemma \ref{L3.1} and \ref{L3.2} rely on the differentiability of the solutions respect to time and spatial variables. By following the standard approximation argument in \cite[Section 2.1.3]{bastin2016stability}, the deduction can be showed valid for $L^2$-solutions.
\end{remark}

\begin{proof} [\textbf{Proof of Theorem \ref{T3.1}}]
	From Lemma \ref{L3.1} and \ref{L3.2}, it follows directly that
	\begin{equation*}
		\begin{aligned}
			&\frac{d}{dt}L(t)+\alpha \mathcal{BC}_1+\mathcal{BC}_2\\ \leq& -\frac{C_1}{2}\| u(t,\cdot)\|^2+(-2\lambda \alpha +C_2-C_1)\| q(t,\cdot)\|^2.		
		\end{aligned}	
	\end{equation*}
	Choose $\alpha$ large enough so that $(-2\lambda \alpha +C_2-C_1)$ is negative. Then we deduce from $f=P^{-1}\begin{pmatrix}
		u \\
		q
	\end{pmatrix}$ and (\ref{3.1}) that
	\vspace{-0.5cm}
	\begin{equation*}
		\begin{aligned}
			&\frac{d}{dt}L(t)+\alpha \mathcal{BC}_1+\mathcal{BC}_2 \\ \leq -&\min(2\lambda \alpha -C_2+C_1,\frac{C_1}{2})\left[\| u(t,\cdot)\|^2+\| q(t,\cdot)\|^2\right] \\ \leq -&\tilde{C} \| f(t,\cdot)\|^2 \\\leq-&\frac{\tilde{C}}{\lambda_{\max}}L(t).
		\end{aligned}
	\end{equation*}
	Here $\tilde{C}$ is a positive constant. Thus Theorem \ref{T3.1} follows provided that 
	\begin{equation} \label{3.3}
		\begin{aligned}
			\mathcal{BC}:=&\alpha \mathcal{BC}_1+\mathcal{BC}_2\\=&\int_{\partial \Omega} f^T\left[\alpha \Lambda_0+\exp\left(-\sum_{j=1}^d \Lambda_jx_j\right) \right]\\&\left(\sum_{j=1}^dn_j(x) \Lambda_j\right) f d\sigma \geq 0.
		\end{aligned}
	\end{equation}
	Note that the matrix in the brackets is diagonal and positive definite. Thus the integrand corresponding to the incoming variables is non-positive while that corresponding to the outgoing variables is non-negative. Hence, the inequality (\ref{3.3}) holds true at least in the trivial case that the incoming variables are set to be zero. 
\end{proof}

To have a general rule for control designs, we refer to the inequality (\ref{3.3}) and divide the boundary $\partial \Omega$ into several parts. Notice that 
$$
\sum_{j=1}^dn_j(x) \Lambda_j=\diag\left(\sum_{j=1}^dn_j(x) u_{1j},\cdots,\sum_{j=1}^dn_j(x) u_{nj}\right).
$$
For $i=1,\cdots,n$, we set 
\begin{equation*}
	\begin{aligned}
		\Gamma_i^+&=\{x \in \partial \Omega:\sum_{j=1}^d n_j(x)u_{ij}>0\},\\
		\Gamma_i^-&=\{x \in \partial \Omega:\sum_{j=1}^d n_j(x)u_{ij}<0\},\\
		\Gamma_i^0&=\{x \in \partial \Omega:\sum_{j=1}^d n_j(x)u_{ij}=0\}.
	\end{aligned}
\end{equation*}
Clearly we have 
$$
\partial \Omega=\Gamma_i^+ \cup \Gamma_i^- \cup \Gamma_i^0.
$$ 
Note that the $i$-th component $f_i$ of $f$ is the incoming (outgoing) variable on $\Gamma_i^-$ ($\Gamma_i^+$).
Thus we recall $\Lambda_0=\diag(\frac{1}{f_1^e},\cdots,\frac{1}{f_n^e})$ and rewrite the left-hand side of the inequality (\ref{3.3}) as
\begin{equation*} 
	\begin{aligned}
		&\sum_{i=1}^n \int_{\Gamma_i^+} (\frac{\alpha}{f_i^e}+\exp(-\sum_{j=1}^d u_{ij}x_j))(\sum_{j=1}^d n_j(x)u_{ij})f_i^2 d\sigma\\+&\sum_{i=1}^n \int_{\Gamma_i^-} (\frac{\alpha}{f_i^e}+\exp(-\sum_{j=1}^d u_{ij}x_j))(\sum_{j=1}^d n_j(x)u_{ij})f_i^2 d\sigma.
	\end{aligned}
\end{equation*}
Hence we have the following rule. 
\begin{corollary} \label{C1}
	The exponential decay estimate in Theorem \ref{T3.1} holds if
	\begin{equation*}
		\begin{aligned}
			&\sum_{i=1}^n \int_{\Gamma_i^+} (\frac{\alpha}{f_i^e}+\exp(-\sum_{j=1}^d u_{ij}x_j))(\sum_{j=1}^d n_j(x)u_{ij})f_i^2d\sigma \\
			\geq&-\sum_{i=1}^n \int_{\Gamma_i^-} (\frac{\alpha}{f_i^e}+\exp(-\sum_{j=1}^d u_{ij}x_j))(\sum_{j=1}^d n_j(x)u_{ij})f_i^2 d\sigma.
		\end{aligned}
	\end{equation*}
\end{corollary}
\begin{remark}
	Corollary 2 presents a general rule to design local or non-local boundary conditions for the purpose of boundary stabilization. As indicated in Section IV below, there are infinitely many suitable boundary conditions, which assign the incoming variables in terms of the outgoing variables, satisfying the rule. For multi-dimensional problems, it seems impossible to express general boundary conditions, especially non-local ones, in a compact  form. Specific examples can be found in the next section, where the control variables are explained.
\end{remark}

From the proofs above, we can see that the functional $L(t)$ can be chosen as
$$
\int_{\Omega}  f^T(t,x)\exp\left(-\sum_{j=1}^d \Lambda_jx_j\right) f(t,x) dx
$$
when $r=0.$ On the other hand, it can be 
$$
\int_{\Omega}  f^T(t,x) \Lambda_0 f(t,x) dx
$$
for $r=n$.

\section{Applications to the 2-D coplanar model}
In the previous section, we have shown that the boundary stabilization can be achieved for discrete velocity models defined on a bounded domain and satisfying Assumptions \ref{A2.1}. Here, we present workable boundary conditions or control laws for the coplanar model \cite{Gatignol1975,platkowski1988discrete} in two dimensions.
We will see that the choices of proper boundary conditions are numerous in multi-dimensional cases, which is different from the spatially one-dimensional case where the boundary usually consists of two endpoints.

Consider the gas confined in a square container $\Omega=(0,1)\times (0,1)$ and postulate that the gas particles move with one of the four velocities of equal positive modulus $U$:
$$
u_1=(U,0),\ \ u_2=(-U,0),\ \ u_3=(0,U),\  \ u_4=(0,-U).
$$
Denote the number density functions by $\tilde{f}=(\tilde{f}_1,\tilde{f}_2,\tilde{f}_3,\tilde{f}_4)^T.$ Here $\tilde{f}_i=\tilde{f}_i(t,x,y)>0$  for $(t,x,y) \in [0,\infty) \times \bar{\Omega}$ is corresponding to the velocity $u_i$ ($i=1,\cdots,4$). The governing equation for each $\tilde{f}_k$ is
\begin{equation} \label{4.1}
	\begin{aligned}	\partial_t\tilde{f}_1+U\partial_x\tilde{f}_1&=\sigma(\tilde{f}_3\tilde{f}_4-\tilde{f}_1\tilde{f}_2),\\
		\partial_t\tilde{f}_2-U\partial_x\tilde{f}_2&=\sigma(\tilde{f}_3\tilde{f}_4-\tilde{f}_1\tilde{f}_2),\\
		\partial_t\tilde{f}_3+U\partial_y\tilde{f}_3&=-\sigma(\tilde{f}_3\tilde{f}_4-\tilde{f}_1\tilde{f}_2),\\
		\partial_t\tilde{f}_4-U\partial_y\tilde{f}_4&=-\sigma(\tilde{f}_3\tilde{f}_4-\tilde{f}_1\tilde{f}_2).
	\end{aligned}
\end{equation}
The above system is called coplanar model, which can be applied to the study of the stationary plane flow around a wedge \cite{cabannes1976etude}.
The symmetry (\ref{2.1}) is clear: $A_{12}^{34}=A_{21}^{34}=A_{34}^{12}=A_{34}^{21}=\sigma>0$ and  other coefficients are $0.$ Furthermore, Assumption \ref{A2.1} is fulfilled since $U>0.$ 

For a steady state $f_e=(f_{1}^e,f_{2}^e,f_{3}^e,f_{4}^e)^T$ with positive components, we denote the deviation $f=\tilde{f}-f_e.$ Then we linearize (\ref{4.1}) at the steady state as in Section \ref{section2}:
\begin{equation} \label{4.5}
	\begin{aligned}
		\partial_t f+\Lambda_1 \partial_x f+\Lambda_2 \partial_y  f= -\mathcal{L}(f_e)\Lambda_0f,
	\end{aligned}
\end{equation}
with $\Lambda_1=\diag(U,-U,0,0), \Lambda_2=\diag(0,0,U,-U), \Lambda_0=\diag(\frac{1}{f_{1}^e},\frac{1}{f_{2}^e},\frac{1}{f_{3}^e},\frac{1}{f_{4}^e}),$ and 
$$
\mathcal{L}(f_e)=\sigma\begin{pmatrix}
	f_{1}^ef_{2}^e   & f_{1}^ef_{2}^e  & -f_{3}^ef_{4}^e  & -f_{3}^ef_{4}^e \\  f_{1}^ef_{2}^e   & f_{1}^ef_{2}^e  & -f_{3}^ef_{4}^e  & -f_{3}^ef_{4}^e \\ -f_{1}^ef_{2}^e   & -f_{1}^ef_{2}^e  & f_{3}^ef_{4}^e  & f_{3}^ef_{4}^e\\ -f_{1}^ef_{2}^e   & -f_{1}^ef_{2}^e  & f_{3}^ef_{4}^e  & f_{3}^ef_{4}^e
\end{pmatrix}.
$$ 
Notice that the components of $f_e$ satisfying $f_{1}^ef_{2}^e=f_{3}^ef_{4}^e,$ thus $\mathcal{L}(f_e)$ is symmetric positive semi-definite with a three-dimensional null space ($r=1$). Furthermore, there exists a positive constant $\alpha$ so that the corresponding functional in Section \ref{section3} decays exponentially provided that the boundary term 
\begin{equation}  \label{4.2}
	\begin{aligned}
		\mathcal{BC}=&U\int_0^1 [-(\frac{\alpha}{f_3^e}+1)f_3^2+(\frac{\alpha}{f_4^e}+1)f_4^2]\big|_{y=0} dx\\[4mm] 
		&+U\int_0^1 [(\frac{\alpha}{f_3^e}+e^{-U})f_3^2-(\frac{\alpha}{f_4^e}+e^{U})f_4^2]\big|_{y=1} dx \\[4mm]
		&+U\int_0^1 [-(\frac{\alpha}{f_1^e}+1)f_1^2+(\frac{\alpha}{f_2^e}+1)f_2^2]\big|_{x=0} dy\\[4mm] &+U\int_0^1 [(\frac{\alpha}{f_1^e}+e^{-U})f_1^2-(\frac{\alpha}{f_2^e}+e^{U})f_2^2]\big|_{x=1} dy
	\end{aligned}
\end{equation}
is non-negative.
Note that the terms with the minus (resp. plus) sign correspond to the incoming (resp. outgoing) variables. 

To proceed, we refer to \cite{ZEISEL2000233} and make the following assumptions:
\begin{itemize}
	\item [(i)]  Gas valves are located on the boundary of the container ($\Omega=(0,1)\times(0,1)$) and can be used to regulate the density of gas inflows.
	\item [(ii)] Gas density monitors are located on the boundary of the container ($\Omega=(0,1)\times(0,1)$) and can be used to measure the  density of gas outflows.
\end{itemize}
According to the direction of the $i$-th mesoscopic velocity $u_i$ $(i=1,2,3,4),$ the number density $\tilde{f}_i=f_i+f_i^e$ corresponds to the gas particles moving inwards (the inflow) from the left, right, bottom and top edges of the container, respectively.

First, we consider the simplest boundary condition that we point out in the proof of Theorem \ref{T3.1}:
\begin{equation}
	\begin{aligned}
		f_1(t,0,y)&=0, \quad y \in (0,1), \\
		f_2(t,1,y)&=0, \quad y \in (0,1), \\
		f_3(t,x,0)&=0, \quad x \in (0,1), \\
		f_4(t,x,1)&=0, \quad x \in (0,1). 
	\end{aligned}
\end{equation}
These boundary conditions mean that the inflow $\tilde{f}_i$ is regulated with gas valves to be the constant level $f_i^e$.
In this case, the boundary term (\ref{4.2}) becomes 
\begin{equation*}
	\begin{aligned}
		\mathcal{BC}=&U\int_0^1 [(\frac{\alpha}{f_4^e}+1)f_4^2]\big|_{y=0} dx\\[4mm]& 
		+U\int_0^1 [(\frac{\alpha}{f_3^e}+e^{-U})f_3^2]\big|_{y=1} dx 
		\\[4mm]&+U\int_0^1 [(\frac{\alpha}{f_2^e}+1)f_2^2]\big|_{x=0} dy\\[4mm]& +U\int_0^1 [(\frac{\alpha}{f_1^e}+e^{-U})f_1^2]\big|_{x=1} dy \geq 0.
	\end{aligned}
\end{equation*}

Now we consider more complicated boundary conditions. Except for an interval $I=\{(x,y) \in (1/3,2/3) \times \{0\} \}$ of the bottom edge, we still use the zero boundary conditions 
\begin{equation}  \label{4.3}
	\begin{aligned}
		f_1(t,0,y)&=0, \quad y \in (0,1), \\
		f_2(t,1,y)&=0, \quad y \in (0,1), \\
		f_3(t,x,0)&=0, \quad x \in (0,\frac{1}{3}) \cup (\frac{2}{3},1), \\
		f_4(t,x,1)&=0, \quad x \in (0,1). 
	\end{aligned}
\end{equation}
For the interval, we consider the following two kinds of control laws.

The first one is to assign the value of  the inflow of this interval in terms of the on-line measurements of the outflow of the left edge:
\begin{align}
	\label{4.4a}    f_3(t,x,0)=k_1f_2(t,0,3x-1), \quad x\in (\frac{1}{3},\frac{2}{3}),\tag{16a}
\end{align}
with $k_1$ a constant to be chosen. Such a constant is our control variable.
According to the assumption (ii), such a boundary condition is realistic. In this case, the boundary term $\mathcal{BC}$ is non-negative provided that
\begin{equation*} 
	\begin{aligned}
		&-U\int_{1/3}^{2/3} (\frac{\alpha}{f_3^e}+1)f_3^2(t,x,0) dx\\[4mm]&+ U\int_0^1 (\frac{\alpha}{f_2^e}+1)f_2^2(t,0,y) dy \\[4mm]
		=&-\frac{U}{3}\int_{0}^{1} (\frac{\alpha}{f_3^e}+1)k_1^2f_2^2(t,0,y) dy\\[4mm] &+ U\int_0^1 (\frac{\alpha}{f_2^e}+1)f_2^2(t,0,y) dy \geq 0.
	\end{aligned}
\end{equation*}
This is true if 
$$
|k_1|\leq \sqrt{\frac{3f_3^e(\alpha+f_2^e)}{f_2^e(\alpha+f_3^e)}}.
$$

Besides (\ref{4.4a}), $f_3(t,x,0)$ can also be specified in the following form 
\begin{align}
	\label{4.4b}    f_3(t,x,0)=k_2f_2(t,0,3x-1)+k_3f_4(t,x,0), \quad x\in (\frac{1}{3},\frac{2}{3}), \tag{16b}
\end{align}
\setcounter{equation}{16}
with $k_2,k_3$ two constants.
This control law assigns the value of the inflow in terms of the on-line measurements of the outflows of the left edge and the interval itself. 

For (\ref{4.4b}), the boundary term $\mathcal{BC}$ is non-negative if 
\begin{equation*}
	\begin{aligned}
		&-\frac{2U}{3}\int_{0}^{1} (\frac{\alpha}{f_3^e}+1)k_2^2f_2^2(t,0,y) dy\\[4mm]&-2U\int_{1/3}^{2/3} (\frac{\alpha}{f_3^e}+1)k_3^2f_4^2(t,x,0) dx\\[4mm]&+U\int_{0}^{1} (\frac{\alpha}{f_4^e}+1)f_4^2(t,x,0) dx\\[4mm] &+ U\int_0^1 (\frac{\alpha}{f_2^e}+1)f_2^2(t,0,y) dy \geq 0.
	\end{aligned}
\end{equation*}
This is true if
\begin{equation*}
	\begin{aligned}
		|k_2|\leq \sqrt{\frac{3f_3^e(\alpha+f_2^e)}{2f_2^e(\alpha+f_3^e)}}, \qquad 
		|k_3| \leq \sqrt{\frac{f_3^e(\alpha+f_4^e)}{2f_4^e(\alpha+f_3^e)}}.
	\end{aligned}
\end{equation*}
In conclusion, both the control laws  (\ref{4.4a}) and (\ref{4.4b}) together with  (\ref{4.3}) ensure that the boundary term $\mathcal{BC}$ is non-negative. Consequently, boundary stabilization is achieved for the two-dimensional linearized coplanar model. Clearly, there are infinitely many feasible control laws and one can choose appropriate forms according to the practical situation.

Finally, we show the effectiveness of the derived control laws (\ref{4.3}) and (\ref{4.4b}) with numerical simulations. In (\ref{4.1}) and (\ref{4.5}), we take  $U=1$,  $\sigma=0.1$ , $f_1^e=4 \ $, $f_2^e=3$, $f_3^e=2$ and $f_4^e=6$. Clearly, $f_1^ef_2^e=f_3^ef_4^e$ holds. The tuning parameters in control laws (\ref{4.4b}) are chosen as $k_2=k_3=0.1.$ The first-order upwind scheme is used to discretize the equation (\ref{4.5}) with mesh sizes  $\Delta x=\Delta y=0.01,$ and the time step  $\Delta t=0.002.$ The time evolution of the logarithm of $L^2$-norm of solutions $(f_1,f_2,f_3,f_4)$ starting from the initial value $(1,1,1,1)$ is plotted with solid line in Figure 1. It shows that the solution to the linearized systems with data given above converges to the origin exponentially.
\begin{figure}
	\begin{center}
		\includegraphics[width=7cm]{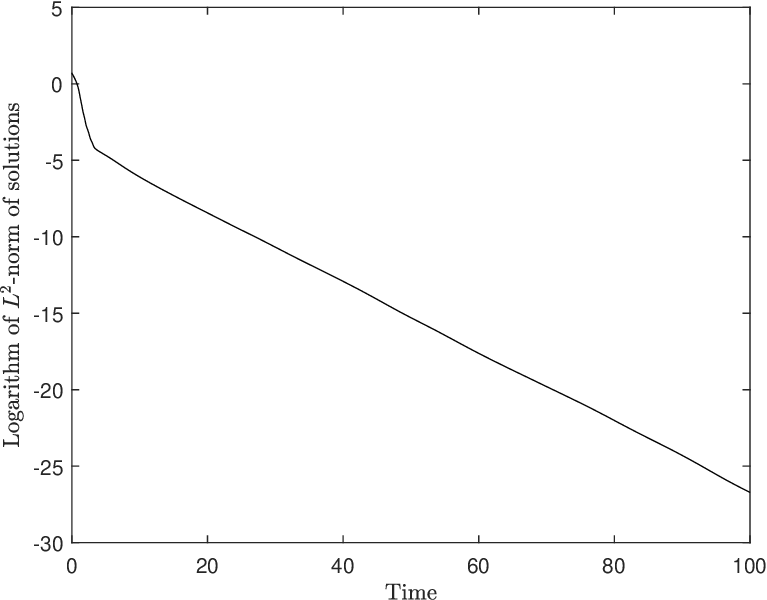}     
		\label{figure2}
		\caption{Time evolution of logarithm of $L^2$-norm of the solution when applying the derived control laws.}                       
	\end{center}                       
\end{figure}

\section{Concluding remarks}
This technical note is concerned with boundary stabilization of multi-dimensional discrete-velocity kinetic models linearized at a uniform steady state. By
exploiting a physically relevant dissipation structure of the models, we design an appropriate Lyapunov function and obtain a boundary control rule. This rule allows for infinitely many choices of boundary conditions ensuring the exponential decay. As examples, specific choices are presented for the two-dimensional coplanar model and verified numerically in Section IV.  
% \textbf{对多维离散速度模型，基于方程内在的dissipative稳定性结构，通过设计适当的Lyapunov，给出了边界稳定化的general rules，该rules允许无限多使得边界稳定的bc选择。我们通过coplanar给出了一些具体的选择。the drawback is close to the uniform ss. It is our ongoing work to extend to extend the present results to semi-linear cases and non-uniform steady state. }

The result above is our first step to design boundary control laws for the multi-dimensional discrete-velocity kinetic models, which are semi-linear hyperbolic systems with relaxation. It is our ongoing work to deal with non-uniform steady states and the nonlinear systems.

\appendix

In this supplemental material, we present a well-posedness theory of $L^2$-solutions to initial-boundary-value problems of the first-order PDE system
\begin{equation} \label{1}
	f_t+\sum_{j=1}^d A_jf_{x_j} =Qf,
\end{equation}
defined on $(t,x) = (t,x_1,\cdots,x_d) \in [0,\infty) \times \Omega$,  with initial data
\begin{equation} \label{2}
	f_0=f_0(x).
\end{equation}
Here $f=
f(t,x) \in \mathbb{R}^N$
is the unknown, $A_j \ (j = 1, \cdots, d)$ and $Q$ are constant matrices in 
$\mathbb{R}^{N\times N}$, and $\Omega \subset \mathbb{R}^d$
is a bounded domain with a Lipschitz boundary. 

We assume that this systems is symmetrizably hyperbolic, that is, there exists a  symmetric positive definite matrix $A_0 \in \mathbb{R}^{N\times N}$ such that
$A_0A_j$ are symmetric for $j=1,\cdots,d.$

To specify a boundary condition, we denote by $\mathbf{n}(x)=(n_1(x),\cdots,n_d(x))$  the unit outward normal vector at boundary point $x \in \partial \Omega$. 
Thanks to the hyperbolicity, the matrix $\sum_{j=1}^d n_j(x)A_j$ can be diagonalized with real eigenvalues, i.e. there exists an invertible matrix $\Pi=\Pi(x)$ such that
\begin{equation} \label{3}
	\begin{aligned}
		\Lambda=\Lambda(x)&=\Pi^{-1}(x)\left[\sum_{j=1}^d n_j(x)A_j\right]\Pi(x)\\& =\begin{pmatrix}
			\Lambda_+(x) & 0 & 0 \\
			0  & 0 & 0 \\
			0 & 0 & \Lambda_-(x)
		\end{pmatrix}.
	\end{aligned}
\end{equation}
Here $\Lambda_+(x) \in \mathbb{R}^{p \times p}$ and $\Lambda_-(x)  \in \mathbb{R}^{n \times n}$ are both diagonal; they contain positive and negative eigenvalues of $\sum_{j=1}^d n_j(x)A_j,$ respectively; $p=p(x)$ and $n=n(x)$ are the numbers of the positive and negative eigenvalues, respectively. Note that $p$ and $n$ may depend on the boundary point $x=(x_1,\cdots,x_d).$
With $\Pi=\Pi(x)$ defined in (\ref{3}) we introduce
\begin{equation} \label{4}
	\zeta(t,x)=\begin{pmatrix}
		\zeta_+(t,x) \\
		\zeta_0(t,x) \\
		\zeta_-(t,x)
	\end{pmatrix}:=\Pi^{-1}(x)f(t,x)
\end{equation}
Here $\zeta_+=\zeta_+(t,x) \in \mathbb{R}^p$ is referred to as the outgoing  variable of system (\ref{1}) at point $x \in \partial \Omega$, $\zeta_-=\zeta_-(t,x) \in \mathbb{R}^n$ is referred to as the incoming  variable, and $\zeta_0=\zeta_0(t,x) \in \mathbb{R}^{N-p-n}$.

% Our boundary condition specifies a linear algebraic relation to determine the incoming variables $\zeta_-$, at boundary point $x \in \partial \Omega$,
% \begin{equation} \label{3}
	%     % \zeta_-(t,x)=\sum_{y \in \partial \Omega} K^y\zeta_+(t,y)
	%         \zeta_-(t,x)=K(x)\zeta_+(t,x)
	% \end{equation}
%  in terms of the outgoing variables $\zeta_+$ of other boundary points $y$ such that
% \begin{equation}     \label{4}
	%  \int_{\partial \Omega} f(t,x)^TA_0\left(\sum_{j=1}^d n_j(x) A_j\right)f(t,x) d\sigma \geq 0.
	% \end{equation}
% Note that the summation in (\ref{3}) may be finite if the matrix $K^y$
% is non-zero only for finite $y \in \partial \Omega$.
% Such a boundary condition allows that the incoming variable at $x$ depends on outgoing variables at $y$ different from $x$. See Section IV for an example of the boundary condition  (\ref{3}). 

Our boundary condition specifies the incoming variable $\zeta_-(t,x)$ at each boundary point $x \in \partial \Omega$ in terms of the outgoing variables $\zeta_+(t,y)$, with $y=x$ or not, such that
\begin{equation}     \label{5}
	\int_{\partial \Omega} f^T(t,x)A_0\left(\sum_{j=1}^d n_j(x) A_j\right)f(t,x) d\sigma \geq 0.
\end{equation}
This implicitly given boundary condition may be local or non-local. A simple example of local boundary conditions is
\begin{equation*} 
	\zeta_-(t,x)=K(x)\zeta_+(t,x)
\end{equation*}
with $K(x) \in \mathbb{R}^{n\times p}$. A non-local boundary condition will be given at the end of this material. 

% Here, we only present the local boundary condition for the sake of simplicity.
% Such a boundary condition allows that the incoming variable at $x$ depends on outgoing variables at $y$ different from $x$. See Section IV for an example of the boundary condition  (\ref{3}). 

Let us show that the integrand in (\ref{5}) can be expressed in terms of the incoming and outgoing variables. To do this, we recall the following
fact. 
\begin{lemma}[Lemma 2.1 of \cite{yang2023feedback}] \label{L1}
	For each $x \in \partial \Omega$, there exist three symmetric positive definite matrices $X_+=X_+(x) \in \mathbb{R}^{p\times p}, X_-=X_-(x) \in \mathbb{R}^{n\times n},X_0=X_0(x) \in \mathbb{R}^{(N-p-n)\times (N-p-n)}$ such that 
	\begin{equation*}
		\Pi^{T}(x)A_0\Pi(x)=\begin{pmatrix}
			X_+(x) & 0  & 0\\
			0 & X_0(x) & 0 \\
			0 & 0 & X_-(x)
		\end{pmatrix}.
	\end{equation*}
\end{lemma}
\noindent
Thanks to this Lemma, it follows from (\ref{3}) that
\begin{equation} \label{6}
	\begin{aligned}
		&f^T A_0\left(\sum_{j=1}^dn_j(x)A_j\right)f \\
		=&(\Pi^{-1}f)^{T}\Pi^TA_0\Pi\Pi^{-1}\left(\sum_{j=1}^dn_j(x)A_j\right)\Pi(\Pi^{-1}f)\\
		=&\begin{pmatrix}
			\zeta_+^T & \zeta_0^T & \zeta_-^T
		\end{pmatrix}\begin{pmatrix}
			X_+\Lambda_+ & 0 & 0\\
			0 & 0 & 0 \\
			0 & 0 & X_-\Lambda_-
		\end{pmatrix}\begin{pmatrix}
			\zeta_+ \\
			\zeta_0 \\
			\zeta_-
		\end{pmatrix}\\
		=&\zeta_+^TX_+\Lambda_+\zeta_++\zeta_-^TX_-\Lambda_-\zeta_-.
	\end{aligned}
\end{equation}
Note that $X_+\Lambda_+$ ($X_-\Lambda_-$) is symmetric positive (negative) definite. 
Formula (\ref{6}) also indicates that the inequality (\ref{5}) holds in the trivial case that all the incoming variables vanish.
% , corresponding to $K(x)=0$ in (\ref{5}).

For the initial-boundary-value problem prescribed above,
our well-posedness result will be established by using the Lumer-Phillips  theorem \cite{curtain2012introduction}  of the semigroup theory  by following \cite[Theorem A.1]{bastin2016stability} for one-dimensional problems.
To do this, we introduce the linear operator $A$:
$$
Af:=-\sum_{j=1}^d A_jf_{x_j}+Qf
$$
and recall the following important fact.
% To understand the inequality (\ref{3}) for $f(t,\cdot) \in L^2(\Omega)$, we recall the following important fact, 
\begin{lemma}[Proposition 6.8 of \cite{chazarain2011introduction}, pp.469] \label{L2}
	Assume that $\Omega$ has a Lipschitz boundary. Then there is a positive constant $C$ such that, for all $f \in L^2(\Omega)$ satisfying $Af \in L^2(\Omega),$ the trace $(\sum_{j=1}^d n_j(x)A_j)f|_{\partial \Omega} \in L^{2}(\partial \Omega)$ is well-defined and satisfies the following estimate 
	$$
	\| (\sum_{j=1}^d n_j(\cdot)A_j)f|_{\partial \Omega}\|_{L^{2}(\partial \Omega)} \leq C(\| f\|_{L^2(\Omega)}+\|Af\|_{L^2(\Omega)}).
	$$
\end{lemma}
% ?
Note that the original result is presented for smooth  and non-characteristic boundaries, but a careful look at the arguments shows that it remains valid for  Lipschitz and characteristic boundaries.
% (the corresponding Sobolev spaces defined on the Lipschitz boundary are the same as those defined on the smooth boundary  \cite{agranovich2015sobolev}) 

Thanks to Lemma \ref{L2}, for $f \in L^2(\Omega)$ and $Af \in L^2(\Omega)$ we deduce from (\ref{3}) that
\begin{equation*}
	\begin{aligned}
		&\Pi^{-1}\left(\sum_{j=1}^d n_j(x)A_j\right)f|_{\partial \Omega}\\=&\Pi^{-1}\left(\sum_{j=1}^d n_j(x)A_j\right)\Pi\Pi^{-1}f|_{\partial \Omega}\\
		=&\begin{pmatrix}
			\Lambda_+ & 0 & 0 \\
			0  & 0 & 0 \\
			0 & 0 & \Lambda_-
		\end{pmatrix}
		\begin{pmatrix}
			\zeta_+ \\
			\zeta_0 \\
			\zeta_-
		\end{pmatrix} \\
		=&\begin{pmatrix}
			\Lambda_+ \zeta_+ \\
			0 \\
			\Lambda_- \zeta_-
		\end{pmatrix} \in L^2(\partial \Omega)
	\end{aligned}
\end{equation*}
for $\Pi^{-1}(x)$ is bounded. Moreover, both the incoming variable $\zeta_-$ and the outgoing variable $\zeta_+$ are in $L^2(\partial \Omega)$, for $\Lambda_{\pm}^{-1}(x)$ are bounded. Consequently, it follows from (\ref{6}) that the integrand in (\ref{5}) is integrable. 
% P^{-1},P,\Lambda_{\pm}^{-1} bounded and measurable.

% boundary conditions for $f \in D(A)$ are understood as algebraic relations of elements in $L^2(\partial \Omega).$

Now we are in a position to  define 
\begin{equation*}
	\begin{aligned}
		D(A)=\{&f \in L^2(\Omega):  Af \in L^2(\Omega), f|_{\partial \Omega} \ \text{satisfies the  }\\ &\text{implicitly given boundary condition.} \}
	\end{aligned}
\end{equation*}
and the symmetrizer-weighted inner product:
\begin{equation*}
	\begin{aligned}
		(f,g)_{A_0}:=\int_{\Omega} f(x)^TA_0g(x) dx, \quad \forall f,g \in L^2(\Omega).
	\end{aligned}
\end{equation*}
Corresponding to this inner product, we introduce the adjoint operator of $A$: 
$$
A^*g:=\sum_{i=1}^d A_ig_{x_i}+A_0^{-1}Q^TA_0g.
$$
Thus, for $f,Af,g,A^*g \in L^2(\Omega)$
we can deduce from the divergence theorem that
\begin{equation*}
	\begin{aligned}
		(Af,g)_{A_0}=&\int_{\Omega} (-\sum_{j=1}^d A_jf_{x_j}+Qf)^TA_0gdx\\
		=&\int_{\partial \Omega}f^TA_0\left(\sum_{j=1}^d n_j(x)A_j\right)g d\sigma+(f,A^*g)_{A_0}.
	\end{aligned}
\end{equation*}

% Here, the $(\cdot,\cdot)_{L^2(\partial \Omega)}$ is the usual $L^2$-inner product in  $L^2(\partial \Omega).$ 
% The well-definedness of the boundary integral can be showed as before. 
% since $\zeta_{\pm}$ and $\xi_{\pm}$ are in $L^{2}(\partial \Omega)$, the integration by parts makes sense. 
% 
% \begin{definition}
	%     A boundary condition for $A^*$, which specify $\xi_+$ in terms of $\xi_-$, is  called "dual" if
	%     \begin{equation} \label{7}
		%  \int_{\partial \Omega}f^TA_0\left(\sum_{j=1}^d n_j(x)A_j\right)g d\sigma=0,\quad \forall f \in D(A). 
		% \end{equation}
	% \end{definition}

Furthermore, we introduce
\begin{equation*}
	\begin{aligned}
		D(A^*)=\big\{&g \in L^2(\Omega):A^*g \in L^2(\Omega) \ \text{and} \\[4mm]
		&\int_{\partial \Omega}f^TA_0\left(\sum_{j=1}^d n_j(x)A_j\right)g d\sigma=0\ \\[4mm] &\text{holds for any}  \ f \in D(A) \big\}.
	\end{aligned}
\end{equation*} 
Note that $D(A^*)$ depends on the boundary condition for $f \in D(A).$
With the definition of $D(A^*),$ we have
$$
(Af,g)_{A_0}=(f,A^*g)_{A_0},\quad \forall f \in D(A),\ \forall g \in D(A^*).
$$

Our main result is
\begin{theorem} \label{T1}
	For the initial-boundary-value problem (\ref{1}), (\ref{2})  with $f_0 \in L^2(\Omega)$ and an implicitly given boundary condition satisfying (\ref{5}), if 
	\begin{equation} \label{7}
		\int_{\partial \Omega} g^TA_0(\sum_{j=1}^d n_j A_j)g d\sigma \leq 0,\quad \forall g \in D(A^*),
	\end{equation}  
	then  there exists a unique solution $f \in C([0,\infty);L^2(\Omega)).$
\end{theorem}

% Having completed the aforementioned preparations, we are now in a position to prove our well-posedness theorem.
\begin{proof}
	\textbf{Step 1: $A$ is a densely defined closed operator.} It is easy to see that $D(A)$ is dense in $L^2(\Omega)$ since $C_0^\infty(\Omega)\subset D(A)$. 
	For the closedness, we suppose that there exist a sequence $\{f^k\}_{k\geq 1} \in D(A)$ such that the two sequences $f^k,Af^k$ converge to $f,g$ in $L^2(\Omega)$ respectively. Clearly $Af^k$ converges to $Af$ in the sense of distribution. Consequently, $Af=g$ in the sense of distribution, by uniqueness of the limit in the sense of distribution. Note that $g \in L^2(\Omega)$, hence $Af \in L^2(\Omega)$ and $ Af=g$ in $L^2(\Omega).$ We arrive at $f^k,Af^k$ converge to $f,Af$ in $L^2(\Omega)$ respectively. Applying the estimate in Lemma \ref{L2},
	the incoming and outgoing variables of $f^k|_{\partial \Omega}$ converge to incoming and outgoing variables of $f|_{\partial \Omega}$ respectively in $L^2(\partial \Omega)$, which implies that $f$ satisfies the boundary conditions in $L^2(\partial \Omega)$. Hence $f \in D(A)$ and $A$ is a closed operator.

	\textbf{Step 2: Estimating $(Af,f)_{A_0}$ and $(A^*g,g)_{A_0}$.} For $f \in D(A)$, we deduce from the divergence theorem and condition (\ref{5}) that
	\begin{equation*}
		\begin{aligned}
			(Af,f)_{A_0}=&\int_\Omega (-\sum_{j=1}^d A_j f_{x_j}+Qf)^TA_0fdx \\
			=&-\frac{1}{2}\int_{\partial \Omega} f^TA_0(\sum_{j=1}^d n_j A_j)f d\sigma+\int_{\Omega} f^TQ^TA_0f \\
			\leq& M(f,f)_{A_0}
		\end{aligned}
	\end{equation*}
	with $M$ a positive constant. Similarly, it follows from (\ref{7}) that
	\begin{equation*}
		\begin{aligned}(A^*g,g)_{A_0}=&\int_\Omega (\sum_{j=1}^d A_j g_{x_j}+A_0^{-1}Q^TA_0g)^TA_0g \\
			=&\frac{1}{2}\int_{\partial \Omega} g^TA_0(\sum_{j=1}^d n_j A_j)g d\sigma+\int_{\Omega} g^TA_0Qg \\
			\leq& M(g,g)_{A_0}.
		\end{aligned}
	\end{equation*}

	With the facts established in Steps 1 and 2, we use the  Lumer-Phillips theorem \cite[Corallary 2.2.3]{curtain2012introduction} and conclude that the operator $A$ generates a unique $C_0$-semigroup, which gives a unique solution $f \in C([0,\infty);L^2(\Omega))$ to the initial-boundary-value problem.
\end{proof}

Finally, we justify the condition (\ref{7}) with two examples by explicitly specifying 
boundary conditions. First of all, we define
\begin{equation*}
	\xi(x)=\begin{pmatrix}
		\xi_+(x) \\
		\xi_0(x) \\
		\xi_-(x)
	\end{pmatrix}:=\Pi^{-1}(x)g(x).
\end{equation*}
Note that $\xi_+=\xi_+(x) \in \mathbb{R}^{p}$  and $\xi_-=\xi_-(x) \in \mathbb{R}^{n}$ are the incoming and outgoing variables corresponding to  the adjoint operator $A^*$, respectively.
Thus the integrand in the definition of $D(A^*)$ can be written as, similar to (\ref{6}),

\begin{equation*} 
	\begin{aligned}
		f^TA_0\left(\sum_{j=1}^d n_j(x)A_j\right)g =\zeta_+^TX_+\Lambda_+\xi_++\zeta_-^TX_-\Lambda_-\xi_-.
	\end{aligned}
\end{equation*}

\noindent
\textbf{Example 1. } For the linear local boundary condition corresponding to the operator $A$:
\begin{equation*} 
	\zeta_-(x)=K(x)\zeta_+(x)
\end{equation*}

with $K(x) \in \mathbb{R}^{n\times p},$  the integrand in the definition of $D(A^*)$  vanishes obviously if   
\begin{equation*} 
	\xi_+(x)=L(x)\xi_-(x)
\end{equation*} 
with 
$$
L(x)=-(X_+\Lambda_+)^{-1}K^T(x)(X_-\Lambda_-).
$$
Then it is straightforward to verify the inequalities (\ref{5}) and (\ref{7}) if $K(x)$ is sufficiently small.
\\

\noindent
\textbf{Example 2.}
This example is about non-local boundary conditions 
for the coplanar model, considered in our paper, with $\Omega=(0,1)\times (0,1)$. For this model,  we have $A_0=\Pi(x)=I_4$. 
The non-local boundary condition (14) together with (15b) is
\begin{equation*}  
	\begin{aligned}
		f_1(0,y)&=0, \quad y \in (0,1), \\
		f_2(1,y)&=0, \quad y \in (0,1), \\
		f_4(x,1)&=0, \quad x \in (0,1),\\
		f_3(x,0)&=0, \quad x \in (0,\frac{1}{3}) \cup (\frac{2}{3},1),\\
		f_3(x,0)&=k_2f_2(0,3x-1)+k_3f_4(x,0), \quad x\in (\frac{1}{3},\frac{2}{3}),
	\end{aligned}
\end{equation*}
with $k_2$ and $k_3$ two constants to be determined so that the inequality (\ref{5}) holds. 

For this non-local boundary condition, the integral in the definition of $D(A^*)$ becomes
\begin{equation*}
	\begin{aligned}
		\mathcal{I}:=&-\int_0^1 f_1g_1|_{x=0}dy+\int_0^1f_2g_2|_{x=0}dy\\
		&+\int_0^1 f_1g_1|_{x=1}dy-\int_0^1f_2g_2|_{x=1}dy\\
		&+\int_0^1 f_3g_3|_{y=1}dx-\int_0^1f_4g_4|_{y=1}dx \\
		&-\int_0^1 f_3g_3|_{y=0}dx+\int_0^1f_4g_4|_{y=0}dx\\
	\end{aligned}
\end{equation*}
\begin{equation*}
	\begin{aligned}
		=&\int_0^1f_2g_2|_{x=0}dy
		+\int_0^1 f_1g_1|_{x=1}dy\\
		&+\int_0^1 f_3g_3|_{y=1}dx 
		-\int_{1/3}^{2/3} f_3g_3|_{y=0}dx+\int_0^1f_4g_4|_{y=0}dx\\
		=&\int_0^1f_2(0,y)\left(g_2(0,y)-\frac{k_2}{3}g_3(\frac{y+1}{3},0)\right)dy
		\\&+\int_{1/3}^{2/3} f_4(x,0)\left(g_4(x,0)-k_3g_3(x,0)\right)dx\\
		&+\int_{[0,1]\backslash (\frac{1}{3},\frac{2}{3})} f_4g_4|_{y=0}dx+\int_0^1 f_1g_1|_{x=1}dy \\
		&+\int_0^1 f_3g_3|_{y=1}dx 
		.
	\end{aligned}
\end{equation*}
Note that the outgoing variables $f_1|_{x=1},f_2|_{x=0},f_3|_{y=1},f_4|_{y=0}$
% $f_1(1,y),f_2(0,y),f_3(x,1),f_4(x,0)$ 
can take any value, then the "dual" boundary condition of $g$ is
\begin{equation*}  
	\begin{aligned}
		g_2(0,y)&=\frac{k_2}{3}g_3(\frac{y+1}{3},0), \quad y \in (0,1), \\
		g_1(1,y)&=0, \quad y \in (0,1), \\
		g_3(x,1)&=0, \quad x \in (0,1), \\
		g_4(x,0)&=0, \quad x \in (0,\frac{1}{3})\cup (\frac{2}{3},1),\\
		g_4(x,0)&=k_3g_3(x,0),\quad x \in (\frac{1}{3},\frac{2}{3}),
	\end{aligned}
\end{equation*}
which makes the integral $\mathcal{I}$ vanish.
% Moreover, if  $g$ satisfies the following "dual" boundary condition 
%  \begin{equation*}  
	%   \begin{aligned}
		% g_2(0,y)&=\frac{k_2}{3}g_3(\frac{y+1}{3},0), \quad y \in (0,1), \\
		%     g_1(1,y)&=0, \quad y \in (0,1), \\
		%     g_3(x,1)&=0, \quad x \in (0,1), \\
		%     g_4(x,0)&=0, \quad x \in (0,\frac{1}{3})\cup (\frac{2}{3},1),\\
		%     g_4(x,0)&=k_3g_3(x,0),\quad x \in (\frac{1}{3},\frac{2}{3}),
		% \end{aligned}
	% \end{equation*}
% then the integral 
% \begin{equation*}
	%     \begin{aligned}
		%      \mathcal{I}
		%         =&\int_0^1f_2g_2|_{x=0}dy 
		%         -\int_{1/3}^{2/3} f_3g_3|_{y=0}dx+\int_{1/3}^{2/3}f_4g_4|_{y=0}dx\\=&\int_0^1f_2g_2|_{x=0}dy+\int_{1/3}^{2/3}f_4g_4|_{y=0}dx
		%         \\
		%         & 
		%         -\int_{1/3}^{2/3} (k_2f_2(0,3x-1)+k_3f_4(x,0))g_3(x,0)dx\\
		%         =&\int_0^1f_2(0,y)g_2(0,y)dy
		%         -\frac{1}{3}\int_{0}^{1} k_2f_2(0,y)g_3(\frac{y+1}{3},0)dy\\
		%     \end{aligned}
	% \end{equation*}
% vanishes.
With the non-local boundary condition for $f$ and the dual boundary condition for $g$, it is straightforward to verify the two inequalities (\ref{5}) and (\ref{7}) under 
\begin{equation*} 
	|k_2|\leq \sqrt{\frac{3}{2}},\quad |k_3|\leq \sqrt{\frac{1}{2}}.
\end{equation*}


\begin{thebibliography}{10}
	
	\bibitem{AURIOL2016300}
	Jean Auriol and Florent {Di Meglio}.
	\newblock Minimum time control of heterodirectional linear coupled hyperbolic pdes.
	\newblock {\em Automatica}, 71:300--307, 2016.
	
	\bibitem{bastin2016stability}
	Georges Bastin and Jean-Michel Coron.
	\newblock {\em Stability and boundary stabilization of 1-d hyperbolic systems}, volume~88.
	\newblock Springer, 2016.
	
	\bibitem{BASTIN201766}
	Georges Bastin and Jean-Michel Coron.
	\newblock A quadratic lyapunov function for hyperbolic density–velocity systems with nonuniform steady states.
	\newblock {\em Systems \& Control Letters}, 104:66--71, 2017.
	
	\bibitem{Serre2006}
	Sylvie Benzoni-Gavage and Denis Serre.
	\newblock {\em {Multi-dimensional hyperbolic partial differential equations: First-order systems and applications}}.
	\newblock Oxford University Press, 11 2006.
	
	\bibitem{cabannes1976etude}
	Henri Cabannes.
	\newblock Etude de l'{\'e}coulement autour d'un di{\`e}dre pour un gaz {\`a} quatre vitesses.
	\newblock {\em Annali di Matematica Pura ed Applicata}, 108:19--40, 1976.
	
	\bibitem{carleman1957problemes}
	Torsten Carleman.
	\newblock Problemes math{\'e}matiques dans la th{\'e}orie cin{\'e}tique des gaz.
	\newblock {\em Publ. Scient. Inst. Mittag-Leftler}, 1957.
	
	\bibitem{chazarain2011introduction}
	Jacques Chazarain and Alain Piriou.
	\newblock {\em Introduction to the theory of linear partial differential equations}.
	\newblock Elsevier, 2011.
	
	\bibitem{coron2022lyapunov}
	Jean-Michel Coron and Hoai-Minh Nguyen.
	\newblock Lyapunov functions and finite-time stabilization in optimal time for homogeneous linear and quasilinear hyperbolic systems.
	\newblock {\em Annales de l'Institut Henri Poincar{\'e} C}, 39(5):1235--1260, 2022.
	
	\bibitem{curtain2012introduction}
	Ruth~F Curtain and Hans Zwart.
	\newblock {\em An introduction to infinite-dimensional linear systems theory}, volume~21.
	\newblock Springer Science \& Business Media, 2012.
	
	\bibitem{de2022backstepping}
	Gustavo~A. de~Andrade, Rafael Vazquez, Iasson Karafyllis, and Miroslav Krstic.
	\newblock Backstepping control of a hyperbolic pde system with zero characteristic speed states.
	\newblock {\em IEEE Transactions on Automatic Control}, pages 1--8, 2024.
	
	\bibitem{DEUTSCHER201754}
	Joachim Deutscher.
	\newblock Finite-time output regulation for linear 2×2 hyperbolic systems using backstepping.
	\newblock {\em Automatica}, 75:54--62, 2017.
	
	\bibitem{DIAGNE2012109}
	Ababacar Diagne, Georges Bastin, and Jean-Michel Coron.
	\newblock Lyapunov exponential stability of 1-d linear hyperbolic systems of balance laws.
	\newblock {\em Automatica}, 48(1):109--114, 2012.
	
	\bibitem{Gatignol1975}
	R.~Gatignol.
	\newblock {\em Theorie Cinetique Des Gaz a Repartition Discrete de Vitesses}.
	\newblock Springer, 1975.
	
	\bibitem{gatignol1975kinetic}
	Ren{\'e}e Gatignol.
	\newblock Kinetic theory for a discrete velocity gas and application to the shock structure.
	\newblock {\em The Physics of Fluids}, 18(2):153--161, 1975.
	
	\bibitem{HAYAT201952}
	Amaury Hayat and Peipei Shang.
	\newblock A quadratic lyapunov function for saint-venant equations with arbitrary friction and space-varying slope.
	\newblock {\em Automatica}, 100:52--60, 2019.
	
	\bibitem{herty2022stabilization}
	Michael Herty and Ferdinand Thein.
	\newblock Stabilization of a multi-dimensional system of hyperbolic balance laws.
	\newblock {\em Mathematical Control and Related Fields}, 2022.
	
	\bibitem{HERTY201612}
	Michael Herty and Wen-An Yong.
	\newblock Feedback boundary control of linear hyperbolic systems with relaxation.
	\newblock {\em Automatica}, 69:12--17, 2016.
	
	\bibitem{higdon1986initial}
	Robert~L Higdon.
	\newblock Initial-boundary value problems for linear hyperbolic system.
	\newblock {\em SIAM review}, 28(2):177--217, 1986.
	
	\bibitem{Hu2016}
	Long Hu, Florent Di~Meglio, Rafael Vazquez, and Miroslav Krstic.
	\newblock Control of homodirectional and general heterodirectional linear coupled hyperbolic pdes.
	\newblock {\em IEEE Transactions on Automatic Control}, 61(11):3301--3314, 2016.
	
	\bibitem{Hu2019}
	Long Hu, Rafael Vazquez, Florent~Di Meglio, and Miroslav Krstic.
	\newblock Boundary exponential stabilization of 1-dimensional inhomogeneous quasi-linear hyperbolic systems.
	\newblock {\em SIAM Journal on Control and Optimization}, 57(2):963--998, 2019.
	
	\bibitem{inamuro1990numerical}
	Takaji Inamuro and Bradford Sturtevant.
	\newblock Numerical study of discrete-velocity gases.
	\newblock {\em Physics of Fluids A: Fluid Dynamics}, 2(12):2196--2203, 1990.
	
	\bibitem{li2010strong}
	Tatsien Li and Bopeng Rao.
	\newblock Strong (weak) exact controllability and strong (weak) exact observability for quasilinear hyperbolic systems.
	\newblock {\em Chinese Annals of Mathematics, Series B}, 31(5):723--742, 2010.
	
	\bibitem{Majda1975InitialboundaryVP}
	Andrew~J. Majda and S.~Osher.
	\newblock Initial‐boundary value problems for hyperbolic equations with uniformly characteristic boundary.
	\newblock {\em Communications on Pure and Applied Mathematics}, 28:607--675, 1975.
	
	\bibitem{PhysRev.37.405}
	Lars Onsager.
	\newblock Reciprocal relations in irreversible processes. i.
	\newblock {\em Phys. Rev.}, 37:405--426, Feb 1931.
	
	\bibitem{PhysRev.38.2265}
	Lars Onsager.
	\newblock Reciprocal relations in irreversible processes. ii.
	\newblock {\em Phys. Rev.}, 38:2265--2279, Dec 1931.
	
	\bibitem{palczewski1997consistency}
	Andrzej Palczewski, Jacques Schneider, and Alexandre~V Bobylev.
	\newblock A consistency result for a discrete-velocity model of the boltzmann equation.
	\newblock {\em SIAM journal on numerical analysis}, 34(5):1865--1883, 1997.
	
	\bibitem{platkowski1988discrete}
	Tadeusz Platkowski and Reinhard Illner.
	\newblock Discrete velocity models of the boltzmann equation: a survey on the mathematical aspects of the theory.
	\newblock {\em SIAM review}, 30(2):213--255, 1988.
	
	\bibitem{Russell1978}
	David~L. Russell.
	\newblock Controllability and stabilizability theory for linear partial differential equations: Recent progress and open questions.
	\newblock {\em SIAM Review}, 20(4):639--739, 1978.
	
	\bibitem{WANG2020104815}
	Ke~Wang, Zhiqiang Wang, and Wancong Yao.
	\newblock Boundary feedback stabilization of quasilinear hyperbolic systems with partially dissipative structure.
	\newblock {\em Systems \& Control Letters}, 146:104815, 2020.
	
	\bibitem{xu2002exponential}
	Cheng-Zhong Xu and Gauthier Sallet.
	\newblock Exponential stability and transfer functions of processes governed by symmetric hyperbolic systems.
	\newblock {\em ESAIM: Control, Optimisation and Calculus of Variations}, 7:421--442, 2002.
	
	\bibitem{yang2023feedback}
	Haitian Yang and Wen-An Yong.
	\newblock Feedback boundary control of multi-dimensional hyperbolic systems with relaxation.
	\newblock {\em Automatica}, 167:111791, 2024.
	
	\bibitem{yong47basic}
	WA~Yong.
	\newblock Basic aspects of hyperbolic relaxation systems, in” advances in the theory of shock waves”, 259--305.
	\newblock {\em Progr. Nonlinear Differential Equations Appl}, 47, 2001.
	
	\bibitem{yong1999singular}
	Wen-An Yong.
	\newblock Singular perturbations of first-order hyperbolic systems with stiff source terms.
	\newblock {\em Journal of differential equations}, 155(1):89--132, 1999.
	
	\bibitem{yong2008interesting}
	Wen-An Yong.
	\newblock An interesting class of partial differential equations.
	\newblock {\em Journal of mathematical physics}, 49(3):033503, 2008.
	
	\bibitem{YONG2019252}
	Wen-An Yong.
	\newblock Boundary stabilization of hyperbolic balance laws with characteristic boundaries.
	\newblock {\em Automatica}, 101:252--257, 2019.
	
	\bibitem{yu2022traffic}
	Huan Yu and Miroslav Krstic.
	\newblock {\em Traffic Congestion Control by PDE Backstepping}.
	\newblock Springer, 2022.
	
	\bibitem{ZEISEL2000233}
	Dieter Zeisel, Hans Menzi, and Ludger Ullrich.
	\newblock A precise and robust quartz sensor based on tuning fork technology for (sf6)-gas density control.
	\newblock {\em Sensors and Actuators A: Physical}, 80(3):233--236, 2000.
	
\end{thebibliography}
\end{document}